\documentclass[11pt]{amsart}

\usepackage{enumerate}

\usepackage{verbatim}

\newcommand{\foot}[1]{}

\newcommand{\hspec}{h_\mathrm{spec}^\perp}
\newcommand{\hexp}{h_\mathrm{exp^+}^\perp}
\newcommand{\uh}{\overline{h}}
\newcommand{\lh}{\underline{h}}
\newcommand{\NN}{\mathbb{N}}

\newcommand{\RR}{\mathbb{R}}
\newcommand{\AAA}{\mathcal{A}}

\newcommand{\CCC}{\mathcal{C}}
\newcommand{\LLL}{\mathcal{L}}
\newcommand{\DDD}{\mathcal{D}}

\newcommand{\SSS}{\mathcal{S}}
\newcommand{\PPP}{\mathcal{P}}
\newcommand{\WWW}{\mathcal{W}}
\newcommand{\Mfe}{\mathcal{M}_f^e}

\newcommand{\GGG}{\mathcal{G}}
\newcommand{\htop}{h_\mathrm{top}}
\newcommand{\ulim}{\varlimsup}
\newcommand{\llim}{\varliminf}
\newcommand{\eps}{\varepsilon}
\newcommand{\symdiff}{\bigtriangleup}
\newcommand{\ph}{\varphi}

\newcommand{\Lamb}{\Lambda}
\newcommand{\NE}{\mathcal{N}^+}

\newtheorem{definition}{Definition}[section]
\newtheorem{theorem}{Theorem}[section]
\newtheorem{lemma}[theorem]{Lemma}
\newtheorem{proposition}[theorem]{Proposition}

\newtheorem{thma}{Theorem}

\theoremstyle{remark}
\newtheorem*{remark}{Remark}

\DeclareMathOperator{\diam}{diam}

\numberwithin{equation}{section}

\begin{document}

\title{Intrinsic ergodicity via obstruction entropies}
\author{Vaughn Climenhaga}
\author{Daniel J. Thompson}
\address{Department of Mathematics, University of Houston, Houston, Texas 77204}
\address{Department of Mathematics, The Ohio State University, 100 Math Tower, 231 West 18th Avenue, Columbus, Ohio 43210}
\email{climenha@math.uh.edu}
\email{thompson@math.osu.edu}
\date{\today}
\thanks{V.C.\ was supported by an NSERC Postdoctoral Fellowship.  D.T.\ is supported by NSF grant DMS-$1101576$}

\begin{abstract}
Bowen showed that a continuous expansive map with specification has a unique measure of maximal entropy.  We show that the conclusion remains true under weaker non-uniform versions of these hypotheses.  To this end, we introduce the notions of obstructions to expansivity and specification, and show that if the entropy of such obstructions is smaller than the topological entropy of the map, then there is a unique measure of maximal entropy.
\end{abstract}

\maketitle

\section{Introduction}

Bowen showed that expansivity and specification imply intrinsic ergodicity~\cite{rB74} -- that is, existence of a unique measure of maximal entropy.  
In~\cite{beta-factor}, we introduced a non-uniform version of specification for a shift space which guarantees intrinsic ergodicity. In this paper, we establish our techniques in a non-symbolic setting. The key idea of ~\cite{beta-factor} is that the obstructions to specification should have less entropy than the whole space. 
Here we adapt this idea to a non-symbolic setting, and 
replace expansivity (which is automatically satisfied by a shift space) with a condition which says that obstructions to positive
 expansivity should have less entropy than the whole space (we work with positive expansivity rather than expansivity for convenience only, see remark after Definition \ref{def:hexp}).


After introducing precise definitions of $\hspec(f)$, which denotes the entropy of obstructions to specification, and $\hexp(f)$, which denotes the entropy of obstructions to expansivity, we establish the following result:

\begin{thma}\label{thm:main0}
Let $X$ be a compact metric space and $f\colon X\to X$ a continuous map.  
If $\hspec(f) < \htop(f)$ and $\hexp(f)<\htop(f)$,
then $f$ is intrinsically ergodic. 
\end{thma}
The obstruction entropies $\hspec(f)$ and $\hexp(f)$ are obtained as limits of quantities $\hspec(f,\epsilon)$ and $\hexp(f,\epsilon)$, which are the entropies of obstructions to specification and expansivity at a fixed finite scale $\epsilon$.  Theorem A is a consequence of the following result, in which we consider obstruction entropies only at suitable fixed scales.
\begin{thma}\label{thm:main}
Let $X$ be a compact metric space and $f\colon X\to X$ a continuous map.  Suppose $\epsilon>0$ is such that $\hspec(f, \epsilon) < \htop(f)$ and $\hexp(f, 28\epsilon) < \htop(f)$.  
Then $f$ is intrinsically ergodic. 
\end{thma}
The hypotheses  of Theorem \ref{thm:main} are a priori weaker and potentially easier to check than the scale-free hypotheses of Theorem \ref{thm:main0}. It is unavoidable that we investigate expansivity at a larger scale than specification. This is due to changes of scale that occur in the course of the proof. In our setting, this cannot be eliminated by standard rescaling techniques of the type used by Bowen \cite{rB74}. (See remark at the end of Section~\ref{sec:spec}.)  This accounts for the (not necessarily optimal) factor of $28$ in the statement of Theorem \ref{thm:main}.



Roughly speaking, $\hexp(f, \epsilon)$ and $\hspec(f, \epsilon)$ may be understood as follows.  
Positive expansivity is equivalent to having $\bigcap_{n \geq 0} B_{n}(x,\epsilon) = \{x\}$ for all $x\in X$ and all sufficiently small $\epsilon>0$.  We define $\hexp(f, \epsilon)$ as the supremum of the entropies of ergodic measures giving positive weight to the set of points at which this condition fails.

Specification at scale $\epsilon$ means that there exists a constant $\tau$ such that every finite collection of finite orbit segments $(x_i,\dots,f^{n_i}x_i)$ can be $\epsilon$-shadowed by a single orbit which takes $\tau$ iterates to transition from one segment to the next.  We look for a collection $\SSS$ of orbit segments (obstructions to specification) such that any finite collection of finite orbit segments can be $\epsilon$-shadowed with gap size $\tau$ \emph{provided we are allowed to first remove elements of $\SSS$ from the ends of each segment}. 
We define $\hspec(f,\epsilon)$ to be the infimum of the entropy of such a collection $\SSS$.

Our goal for this note is to establish these techniques and concepts, particularly the use of $\hexp$. As a first novel application of Theorem \ref{thm:main}, we can establish intrinsic ergodicity for non-symbolic factors of $\beta$-shifts, subject to a weak expansivity condition, extending our results from ~\cite{beta-factor}. This includes the following result:
\begin{thma}\label{thm:beta}
Every positively expansive factor of a $\beta$-shift is intrinsically ergodic.
\end{thma}
This result suggests that arbitrary $\beta$-shifts may have value as coding spaces. The coding of (algebraic) dynamical systems by $\beta$-shifts has been well studied in the special case that $\beta$ is a Pisot number, and hence the corresponding  $\beta$-shift is sofic (see \cite{kS00, nS03, LS} for an extensive list of references).





\section{Definitions}

Let $X$ be a compact metric space and $f\colon X\to X$ a continuous map. We recall some standard definitions, and introduce the definitions of $\hspec$ and $\hexp$.  The notions that follow depend on the dynamics $f$, although our notation suppresses this dependence in order to simplify the presentation. 

\subsection{Entropy}

\begin{definition}
A set $E\subset X$ is \emph{$(n,\eps)$-separated} for some $n\in \NN$ and $\eps>0$ if for all $x\neq y\in E$ we have $y\notin B_n(x,\eps)$, where
\[
B_n(x,\eps) = \{ y\in X \mid d_n(x, y) < \eps \}
\]
and $d_n(x,y) = \max \{ d(f^ix, f^iy) \mid 0 \leq i \leq n-1\}$.
Given $Z\subset X$, we write $\Lamb(Z,n,\eps)$ for the maximum cardinality of a $(n,\eps)$-separated set $E\subset Z$.  
Let
\begin{equation}\label{eqn:uh}
\uh(Z,\eps) := \ulim_{n\to\infty} \frac 1n \log \Lamb(Z,n,\eps), \qquad \uh(Z) = \lim_{\eps\to 0} \uh(Z,\eps),
\end{equation}
and define $\lh$ analogously.  If $\lh(Z,\eps) = \uh(Z,\eps)$ then we write  $h(Z,\eps)$ for the common value  and call this the \emph{topological entropy} of $Z$ at scale $\eps$.
\end{definition}

\begin{remark}
In fact, $\lh$ and $\uh$ are the \emph{lower and upper capacity topological entropies}~\cite{yP97}. If $Z$ is invariant then $\lh(Z,\eps)=\uh(Z,\eps)$ for all $\eps>0$~\cite[Theorem 11.5]{yP97}. In particular, $h(X,\eps)$ exists for every $\eps$, and $h(X)$ is the standard definition of topological entropy for the system $(X,f)$.  For unity of notation, from now on we will usually write $h(X)$ in place of $\htop(f)$.
\end{remark}

The variational principle states that $h(X) = \sup_\mu h_\mu(f)$, where $h_\mu(f)$ is the measure-theoretic entropy of $\mu$ and the supremum is taken over all Borel $f$-invariant probability measures on $X$~\cite{pW82}.  A measure achieving the supremum is a \emph{measure of maximal entropy} (MME), and a system with a unique MME is called \emph{intrinsically ergodic} \cite{beta-factor, bW70, fH79, fH81, jB97}. 
 
We will have occasion to consider the entropy of a \emph{sequence} of sets $Z_n$. 
Such a sequence generates a subset $\DDD \subset X\times \NN$ in a natural way, as $\DDD = \bigcup_n Z_n\times \{n\}$.  Conversely, given $\DDD\subset X\times \NN$ we obtain a sequence of sets $\DDD_n = \{x\in X \mid (x,n)\in \DDD\}$.  We think of subsets $\DDD\subset X\times \NN$ as collections of points and times, so that given $x\in X$, the set $\{n\in \NN \mid (x,n)\in \DDD\}$ can be thought of as those times at which the orbit segment $(x, f(x), \dots, f^n(x))$ satisfies a certain property.

\begin{definition}
Give $\DDD\subset X\times \NN$, write $\Lamb(\DDD,n,\eps) = \Lamb(\DDD_n,n,\eps)$.  The (upper) topological entropy of $\DDD$ is
\begin{equation}\label{eqn:uh2}
\uh(\DDD,\eps) := \ulim_{n\to\infty} \frac 1n \log \Lamb(\DDD,n,\eps), \qquad \uh(\DDD) = \lim_{\eps\to 0} \uh(\DDD,\eps).
\end{equation}
\end{definition}
This procedure of taking the capacity entropy of a sequence of sets appears in the definition of coarse multifractal spectra, and a general discussion of the relationship between this quantity and other dimensional quantities is given in~\cite{pressure-spectrum}.

\subsection{Expansivity} \label{expansive}

Positive expansivity at scale $\eps$ is equivalent to the following property: for every $x\in X$, we have
\begin{equation}\label{eqn:Phi}
\Phi^+_x(\eps) := \bigcap_{n\geq 0} B_n(x,\eps) = \{x\}.
\end{equation}
\begin{definition}
Denote by $\NE(\eps) := \{x\in X \mid \Phi^+_x(\eps) \neq \{x\}\}$ the set of non-expansive points at scale $\eps$.  Following Buzzi and Fisher~\cite{BF11}, we say that an $f$-invariant measure $\mu$ is \emph{almost positively expansive at scale $\eps$} if $\mu(\NE(\eps))=0$; in other words, if $\Phi^+_x(\eps) = \{x\}$ for $\mu$-a.e.\ $x$.  
\end{definition}
\begin{definition}\label{def:hexp}
Writing $\Mfe$ for the set of ergodic $f$-invariant Borel probability measures on $X$, the \emph{entropy of obstructions to positive expansivity at scale $\eps$ } is
\begin{align*}
\hexp(f, \eps) &= \sup \{h_\mu(f) \mid \mu\in\Mfe \text{ is not almost positively expansive} \} \\
&= \sup \{h_\mu(f) \mid \mu\in\Mfe \text{ and } \mu(\NE(\eps))>0 \}.
\end{align*}
We also define $\hexp(f)= \lim_{\eps \to 0} \hexp(f, \eps) $. The limit exists since $\hexp(f, \eps)$ is monotonic when considered as a function of $\eps$.
\end{definition}
\begin{remark}
Positive expansivity is the appropriate definition for non-invertible maps.  For invertible maps, Buzzi and Fisher define almost expansivity using $\Phi_x(\eps) := \bigcap_{n\geq 0} f^n(B_{2n}(f^{-n}x,\eps))$ and requiring that $\mu$ give zero measure to the non-expansive set $\mathcal{N}(\eps) := \{x\in X \mid \Phi_x(\eps) \neq \{x\}\}$. Thus, when $f$ is an expansive homeomorphism on $X$, every $f$-invariant measure is almost-expansive.   We can define the \emph{entropy of obstructions to expansivity at scale $\eps$} for a homeomorphism $f$ as
\begin{align*}
h_{\mathrm{exp}}^\perp (\eps) &= \sup \{h_\mu(f) \mid \mu \text{ is $f$-invariant and not almost expansive} \} \\
&= \sup \{h_\mu(f) \mid \mu \text{ is $f$-invariant and } \mu(\mathcal{N}(\eps))>0 \}.
\end{align*}
We do not consider the invertible case or $h_{\mathrm{exp}}^\perp$ in this note. This is purely for convenience. Our proofs go through with minimal modification in the situation when $f$ is invertible and $h_{\mathrm{exp}}$ is less than the entropy of the space.
\end{remark}

The following proposition gives a useful method for bounding $\hexp$.

\begin{proposition}\label{prop:hexpleq}
$\hexp(f,\eps) \leq \lh(\NE(\eps))$.
\end{proposition}
\begin{proof}
Given $h<\hexp(f,\eps)$, let $\mu\in\Mfe$ be such that $h_\mu(f) > h$ and $\mu(\NE(\eps))>0$.  By the Brin--Katok local entropy formula~\cite{BK83}, there exists $N\in \NN$ and $Y\subset X$ such that $\mu(Y) > 1 - \mu(\NE(\eps))$ and $\mu(B_n(x,\eps)) \leq e^{-nh}$ for all $n\geq N$.  Taking $Z = Y \cap \NE(\eps)$, we get $\mu(Z)>0$.

Let $E_n\subset Z$ be any maximal $(n,\eps)$-separated set.  By maximality, we have $Z\subset \bigcup_{x\in E_n} B_n(x,\eps)$, whence $\mu(Z) \leq (\#E_n) e^{-nh}$.  Thus $\#E_n \geq \mu(Z) e^{nh}$, and it follows that $\lh(\NE(\eps)) \geq \lh(Z,\eps) \geq h$.  Since $h<\hexp(f,\eps)$ was arbitrary, we are done. 
\end{proof}

The following two propositions are crucial to our approach.

\begin{proposition}\label{prop:generating}
If $\mu$ is almost positively expansive at scale $\eps$ and $\AAA$ is a measurable partition of $X$ such that every element of $\AAA$ is contained in $B(x,\eps)$ for some $x\in X$, then $\AAA$ is (one-sided) generating for $\mu$.
\end{proposition}
\begin{proposition}\label{prop:entropy-appeared}
If $\hexp(f,\eps)< h(X)$, then $h(X,\eps) = h(X)$.
\end{proposition}
Proposition \ref{prop:generating} is immediate from the definitions. Before proving Proposition \ref{prop:entropy-appeared}, we introduce a useful concept.

\begin{definition}
Let $E_n\subset X$ be an $(n, \eps)$-separated set of maximum cardinality. We say that a partition $\AAA^n$ is adapted to $E_n$ if $\AAA^n$ is a partition such that each atom $w$ of the partition satisfies
\[
B_n \left(x, \frac\eps 2\right) \subset w \subset B_n (x, \eps)
\]
for some $x \in E_n$.
\end{definition}

The existence of adapted partitions for any $E_n$ follows from the fact that the sets $B_n(x, \eps/2)$ are disjoint and the sets $B_n(x,\eps)$ cover $X$.  For every $n\geq 1$, the non-expansive set for the map $f^n$ with respect to the Bowen metric $d_n = \max_{0\leq k < n} d\circ f^k$ coincides with the non-expansive set for $f$ with respect to the original metric $d$.  Thus, Proposition~\ref{prop:generating} shows that if $\mu$ is almost positively expansive at scale $\eps$ and $\AAA^n$ is an adapted partition for a maximal $(n,\eps)$-separated set, then $\AAA^n$ is generating for $\mu$ under the map $f^n$.

\begin{proof} [Proof of Proposition \ref{prop:entropy-appeared}]
For each $n\geq 1$, let $E_n$ be a maximal $(n,\eps)$-separated set, and $\AAA^n$ an adapted partition for $E_n$.  Fix $\mu\in \Mfe$ with $h_\mu(f) > \hexp(f,\eps)$, noting that such $\mu$ exist by the variational principle. 
By the above remarks, $\AAA^n$ is generating for $\mu$ under the map $f^n$.  In particular, we have
\[
h_\mu(f) = \frac 1n h_\mu(f^n) = \frac 1n h_\mu(f^n,\AAA^n) \leq \frac 1n \log \#\AAA^n = \frac 1n \log \# E_n,
\]
and sending $n\to\infty$ gives $h_\mu(f) \leq h(X,\eps)$.  Taking a supremum over all such $\mu$ and applying the variational principle gives the result.
\end{proof}

We will see in $\S\ref{sec:Gibbs}$ that Proposition \ref{prop:entropy-appeared} implies the existence of an MME.

\subsection{Specification}\label{sec:spec}


\begin{definition}\label{def:spec}
A collection of points and times $\GGG \subset X\times \NN$ has \emph{specification at scale $\eps$} if there exists $\tau\in \NN$ such that for every $\{(x_j,n_j)\mid 0\leq j\leq k\} \subset \GGG$,
\begin{equation}\label{eqn:spec}
\bigcap_{j=0}^{k} f^{-\sum_{i=0}^{j-1} (n_i +\tau)} B_{n_j}(x_j, \eps) \neq \emptyset.
\end{equation}
If the intersection always contains a periodic point with period $\sum_{i=0}^k (n_i + \tau)$, we say that $\GGG$ has \emph{(Per)-specification at scale $\eps$}.
\end{definition}


\begin{remark}
There are two standard definitions of specification for the dynamical system $(X,f)$, which differ from each other only in whether or not one requires that the shadowing orbit be periodic.  Using our terminology, these properties are equivalent to the requirement that the entire set $X\times \NN$ has specification, or (Per)-specification, for every $\eps>0$.
\end{remark}


\begin{definition}\label{def:decomp}
A triple $(\PPP,\GGG,\SSS)\subset (X\times \NN)^3$ is a \emph{decomposition} for $(X,f)$ if for every $x\in X$ and $n \in \NN$ there exist $p, g, s \in \NN$ with $p+g+s=n$ so that
\[
(x,p)\in \PPP, \qquad (f^px,g)\in \GGG, \qquad (f^{p+g}x,s)\in \SSS.
\]
That is, every orbit segment can be decomposed into a `prefix' from $\PPP$, a `good' core from $\GGG$, and a `suffix' from $\SSS$.  Given a decomposition $(\PPP,\GGG,\SSS)$, we will also need to consider the following collections of points and times:
\[
\GGG^M := \{ (x,n) \in X\times \NN \mid p(x,n) \leq M, s(x,n) \leq M\}.
\]
\end{definition}

\begin{remark}
Note that $X\times \NN = \bigcup_M \GGG^M$, so that a decomposition $(\PPP,\GGG,\SSS)$ defines a filtration of the collection of all points and times.  We are interested in the situation when each level $\GGG^M$ of this filtration has specification.  In the course of the proof we will see that every MME $\mu$ has the property that $\inf_n \mu(\GGG_n^M) \to 1$ as $M\to\infty$. This suggests that the collections $\GGG^M$ can be thought of as regular sets analogous to the Pesin sets of non-uniform hyperbolicity theory.
\end{remark}

\begin{definition}\label{def:obstructions}
The \emph{entropy of obstructions to specification at scale $\eps$} is
\begin{multline*}
\hspec(f,\eps) := \inf \{ \uh(\PPP \cup \SSS,3\eps) \mid \PPP,\SSS \subset X\times \NN \text{ are such that }\\
\text{there exists a decomposition $(\PPP,\GGG,\SSS)$ for $(X,f)$} \\
\text{for which every $\GGG^M$ has specification at scale $\eps$} \}.
\end{multline*}
We define $\hspec(f, \eps)= \lim_{\eps \to 0} \hspec(f, \eps) $. The limit exists since $\hspec(f, \eps)$ is monotonic when considered as a function of $\eps$.
\end{definition}

\begin{remark}
A priori, $\GGG$ could have specification at scale $\epsilon$, but not at a smaller scale $\delta$.  Thus, $\hspec(f,\epsilon)$ may increase as $\epsilon$ decreases, 
which is why the hypotheses in Theorem~\ref{thm:main} are a priori weaker than those in Theorem~\ref{thm:main0}. This contrasts with the classical setting of Bowen, where specification at a fixed scale and expansivity at that same scale together imply the specification property at all smaller scales. This can be shown by using expansivity and compactness to find $N=N(\delta,\epsilon)$ such that $B_N(x,\epsilon) \subset B(x,\delta)$ for every $x$, so that specification at scale $\epsilon$ with gap size $\tau$ implies specification at scale $\delta$ with gap size $\tau + N(\delta,\epsilon)$. This argument does not work in our setting, because although such  an $N=N(x,\delta,\epsilon)$ exists for each $x\notin \NE(\epsilon)$, it cannot be chosen uniformly.
\end{remark}

\section{Proofs}

Throughout this section, $(X,d)$ is a compact metric space and $f\colon X\to X$ is continuous.  Large parts of the proof work without the full strength of the hypotheses of Theorem~\ref{thm:main}, so for each partial result we state precisely which conditions are needed.  Our method is inspired by Bowen's original proof from~\cite{rB74}, incorporating the modifications from~\cite{beta-factor} that allow us to use the weakened version of specification.

We outline the strategy of the proof.  First, we obtain a series of combinatorial estimates which yield bounds on the maximum cardinalities of various $(n,\delta)$-separated sets (Lemmas \ref{lem:ssum}-\ref{lem:sleq2}). 
 These estimates allow us to construct a measure of maximal entropy with a certain weak Gibbs property (Lemma \ref{lem:GibbsG}), which we show is ergodic (Proposition \ref{prop:reallyweakmix}). We use the weak Gibbs property, together with a crucial combinatorial estimate on positive measure sets (Lemma \ref{lem:posmeas}), to rule out the existence of any other measure of maximal entropy.

\subsection{Lower bounds on $X$}

We obtain a lower bound on the maximum cardinality of $(n,\delta)$-separated sets for $X$ using a completely general, and rather standard, argument which is essentially contained in Bowen \cite{rB74}.

\begin{lemma}\label{lem:ssum}
$\Lamb(X,\sum_{j=1}^k n_j,2\delta) \leq \prod_{j=1}^k \Lamb(X,n_j,\delta)$ for every $\delta>0$ and $n_j$.
\end{lemma}
\begin{proof}
Let $E$ be maximal $(\sum n_j,2\delta)$-separated and let $E_j$ be maximal $(n_j,\delta)$-separated.  Then $E_j$ is $(n_j,\delta)$-spanning, and we define a map $\pi\colon E\to E_1\times\cdots\times E_k$ by the condition that $f^{n_1+\cdots+n_{j-1}}(x)\in B_{n_j}(\pi_j(x),\delta)$ for each $j$.  This map is injective.
\end{proof}

\begin{lemma}\label{lem:sgeq}
For every $n\in \NN$ and $\delta>0$, we have $\Lamb(X,n,\delta) \geq e^{nh(X,2\delta)}$.
\end{lemma}
\begin{proof}
By Lemma~\ref{lem:ssum} we have $\Lamb(X,kn,2\delta) \leq (\Lamb(X,n,\delta))^k$.  It follows that $\log \Lamb(X,n,\delta)\geq n (\frac 1{kn} \log \Lamb(X,kn,2\delta))$, and we send $k\to\infty$.
\end{proof}

\subsection{Upper bounds on $\GGG$}

Using a standard argument based on the specification property of $\GGG$, we obtain upper bounds on the cardinality of an $(n,\delta)$-separated set in $\GGG_n$.

\begin{lemma}\label{lem:ssumspec}
If $\GGG \subset X\times \NN$ has specification at scale $\delta$ with gap size $\tau$, then
\[
\Lamb(X,n_1+\cdots+n_k + (k-1)\tau, \delta) \geq \prod_{j=1}^k \Lamb(\GGG,n_j,3\delta)
\]
for every $n_1,\dots,n_k$.
\end{lemma}
\begin{proof}
Let $E_j\subset \GGG_{n_j}$ be a $(n_j,3\delta)$-separated set of maximum cardinality, and use the specification property to define a map $\pi\colon E_1\times\cdots E_k \to X$ such that $\pi(x_1,\dots,x_k)$ is contained in the intersection from~\eqref{eqn:spec}.  A computation similar to the proof of Lemma~\ref{lem:ssum} shows that the image of $\pi$ is $(n_1+\cdots+n_k+ (k-1)\tau,\delta)$-separated.
\end{proof}

\begin{lemma}\label{lem:sleq}
If $\GGG \subset X\times \NN$ has specification at scale $\delta$ with gap size $\tau$, then $\Lamb(\GGG,n,3\delta) \leq e^{(n+\tau)h(X,\delta)}$ for every $n\in \NN$.
\end{lemma}
\begin{proof}
Every $(m,\delta)$-separated set is $(n,\delta)$-separated for $n>m$.  Thus Lemma~\ref{lem:ssumspec} implies $\Lamb(X,k(n+\tau),\delta) \geq (\Lamb(\GGG,n,3\delta))^k$, and the result follows by sending $k\to\infty$ in the resulting inequality $\log \Lamb(\GGG,n,3\delta)\leq (n+\tau) (\frac 1{k(n+\tau)} \log \Lamb(X,k(n+\tau),\delta))$.
\end{proof}

\subsection{Lower bounds on $\GGG$ and related sets}

First, we prove the following useful summability result.

\begin{lemma} \label{summability}
If $\PPP,\SSS\subset X\times \NN$ are such that $\uh(\PPP\cup\SSS,\delta)<h(X,\delta)$, then
$\sum_{i\geq M} \Lamb(\PPP\cup\SSS,i,\delta^\prime) e^{-ih(X,\delta)} \to 0$ as $M\to\infty$ for every $\delta^\prime\geq\delta$.
\end{lemma}
\begin{proof}
Write $h:=h(X,\delta)$ and let $\gamma>0$ be such that $\uh(\PPP\cup\SSS,\delta) < h-2\gamma$.  Then there exists a constant $C>0$ such that $\Lambda(\PPP\cup\SSS,n,\delta) < C e^{n(h-\gamma)}$ for all $n$, and so
\[
e^{-nh} \Lamb(\PPP\cup\SSS,n, \delta^\prime) \leq e^{-nh} \Lamb(\PPP\cup\SSS,n, \delta) < Ce^{-nh}e^{n(h-\gamma)} = Ce^{-n\gamma}.
\]
Thus $\sum_{n=0}^\infty e^{-nh} \Lamb(\PPP\cup\SSS,n,\delta^\prime) < C\sum_{n=0}^\infty (e^{-\gamma})^n < \infty$, and the tail of the series converges to $0$.
\end{proof}

The following result holds for decompositions where $\GGG$ has specification, without any conditions on $\GGG^M$ or on expansivity. The result should be interpreted as a lower bound on the growth rate of a `fattened up' version of $\GGG$.

\begin{lemma}\label{lem:lotsinG}
Let $(\PPP,\GGG,\SSS)$ be a decomposition for $(X,f)$ such that
\begin{enumerate}
\item \label{delta-spec} $\GGG$ has specification at scale $\delta$ with gluing time $\tau$, and
\item \label{uh-less} $\uh(\PPP\cup\SSS,\delta) < h(X,\delta)$.
\end{enumerate}
Then for every $\gamma_1, \gamma_2>0$ there exists $M\in \NN$ such that the following is true:  For any $\DDD \subset X\times \NN$ such that $\Lamb(\DDD,n,6\delta) \geq \gamma_1 e^{n h(X,\delta)}$ for all $n$, we have $\Lamb(\DDD \cap \GGG^M,n,6\delta) \geq (1-\gamma_2) \Lamb(\DDD,n,6\delta)$ for all $n$.
\end{lemma}
\begin{proof}
Let $D_n \subset \DDD$ be a $(n,6\delta)$-separated set of maximum cardinality, and given $i,j,k\in \NN$, let
\[
D(i,j,k) = \{x\in D_n \mid n=i+j+k, p(x,n)=i, g(x,n) = j, s(x,n)=k \},
\]
where $p,g,s$ are as in Definition~\ref{def:decomp}.  Observe that $\bigcup\{ D(i,j,k) \mid i,k \leq M \} \subset \DDD \cap \GGG^M$ for every $M$, and in particular
\[
D_n^M := \bigcup\{ D(i,j,k) \mid i+j+k=n, i,k\leq M \}
\]
is an $(n,6\delta)$-separated subset of $\DDD \cap \GGG^M$, so $\Lamb(\DDD \cap\GGG^M,n,6\delta) \geq \#D_n^M$.

Let $E^P_i \subset \PPP_i$ be maximal $(i,3\delta)$-separated, and similarly for $E^G_j \subset \GGG_j$ and $E^S_k \subset \SSS_k$.  As in Lemma~\ref{lem:ssum}, there is an injection $\pi\colon D(i,j,k) \to E^P_i \times E^G_j \times E^S_k$, and we obtain
\[
\Lamb(\DDD,n,6\delta) \leq \sum_{\substack{i+j+k=n \\ i,k\leq M}} \# D(i,j,k) + \sum_{\substack{i+j+k=n \\ i\vee k\geq M}} \Lamb(\PPP,i,3\delta) \Lamb(\GGG,j,3\delta) \Lamb(\SSS,k,3\delta).
\]
The first sum is equal to $\#D_n^M$.  By Condition~\eqref{delta-spec}, Lemma~\ref{lem:sleq} gives
\[
\Lamb(\DDD,n,6\delta) \leq \Lamb(\DDD\cap \GGG^M,n,6\delta) + \sum_{\substack{i+j+k=n \\ i\vee k\geq M}} \Lamb(\PPP,i,3\delta) \Lamb(\SSS,k,3\delta) e^{(j+\tau)h(X,\delta)}.
\]
Write $a_i = \Lamb(\PPP\cup\SSS,i,3\delta) e^{-ih(X,\delta)}$ and observe that by Condition~\eqref{uh-less} and Lemma \ref{summability}, we have $b_M := \sum_{i\geq M} a_i \to 0$ as $M\to\infty$.  Together with the condition on $\DDD$, this yields
\begin{align*}
\Lamb(\DDD,n,6\delta) &\leq \Lamb(\DDD\cap\GGG^M,n,6\delta) + e^{(n+\tau)h(X,\delta)} \sum_{i \vee k\geq M} a_i a_k \\
&\leq \Lamb(\DDD\cap\GGG^M,n,6\delta) + \gamma_1^{-1} \Lamb(\DDD,n,6\delta) e^{\tau h(X,\delta)} b_M^2.
\end{align*}
Choosing $M$ such that $\gamma_1^{-1} e^{\tau h(X,\delta)} b_M^2 < \gamma_2$, the proof is complete.
\end{proof}

\begin{remark}
If we assume the hypotheses of Theorem \ref{thm:main}, then we have $\hspec(f, \delta)< \htop(f) = \htop(f, \delta)$ for every $\delta\in[\epsilon,28\epsilon]$, where the equality uses Proposition~\ref{prop:entropy-appeared}. This gives the existence of a decomposition $(\PPP,\GGG,\SSS)$ for $(X,f)$ so that $\GGG^M$ has specification for each $M$, and $\uh(\PPP\cup\SSS,\delta)<h(X,\delta)$. Thus, the results of Lemmas \ref{summability} and \ref{lem:lotsinG} apply to this decomposition. This is the only place in the proof of Theorem \ref{thm:main} where we use the assumption  $\hspec(f, \epsilon)< \htop(f)$.
\end{remark}


\subsection{Upper bounds on $\Lambda(X, n, 6\delta)$} 
From now on, we assume that $\hspec(f,\delta) < h(X, 
\delta)$.
Thus, there is a decomposition $(\PPP,\GGG,\SSS)$ for $(X,f)$ satisfying
\begin{enumerate}
\item each $\GGG^M$ has specification at scale $\delta$ with gluing time $\tau_M$, and
\item $\uh(\PPP\cup\SSS,\delta) < h(X,\delta)$.
\end{enumerate}
In this section, we also assume that
\begin{equation}\label{eqn:heq}
h(X,12\delta) = h(X,\delta),
\end{equation}
which, by Proposition \ref{prop:entropy-appeared}, is a weaker assumption than $\hexp (f,12 \delta) <h(X)$. 
\begin{lemma}\label{lem:lotsinG2}
Under the above conditions, for every $\gamma>0$ there exists $M$ such that $\Lamb(\GGG^M,n,6\delta) \geq (1-\gamma) \Lamb(X,n,6\delta)$ for all $n$.
\end{lemma}
\begin{proof}
Lemma~\ref{lem:sgeq} together with~\eqref{eqn:heq} gives $\Lamb(X,n,6\delta) \geq e^{nh(X,\delta)}$, so we can apply Lemma~\ref{lem:lotsinG}.
\end{proof}

\begin{lemma}\label{lem:sleq2}
There exists $C_1\in \RR$ such that $e^{nh(\delta)} \leq \Lamb(X,n,6\delta) \leq C_1 e^{nh(\delta)}$ for all $n$.
\end{lemma}
\begin{proof}
As above, the first inequality follows from Lemma~\ref{lem:sgeq} and~\eqref{eqn:heq}.  The second inequality follows by applying Lemma~\ref{lem:lotsinG2} with $\gamma=\frac12$ and then applying Lemma~\ref{lem:sleq} to $\GGG^M$, so that 
\[
\Lamb(X,n,6\delta) \leq 2\Lamb(\GGG^M,n,6\delta) \leq (2e^{\tau_Mh(X,\delta)})e^{nh(X,\delta)}.\qedhere
\]
\end{proof}

\subsection{An MME with a Gibbs property}\label{sec:Gibbs}
We replace~\eqref{eqn:heq} with the stronger assumption that for sufficiently small $\delta >0$, 
\begin{equation}\label{eqn:heq2}
h(X,\delta) = h(X).
\end{equation}
This ensures the existence of an MME, constructed using the following standard argument. 
Let $E_n \subset X$ be a maximal $(n,\delta)$-separated set, and take
\[
\nu_n := \frac{1}{\# E_n}\sum_{x\in E_n} \delta_x.
\]
In order to obtain invariant measures, we consider the measures
\begin{equation}\label{eqn:mun}
\mu_n := \frac{1}{n}\sum_{k=0}^{n-1} f_*^k \nu_n
\end{equation}
and let $\mu$ be a weak* limit of the sequence $\{\mu_n\}$.  

\begin{lemma}\label{lem:MME}
$\mu$ is a measure of maximal entropy.
\end{lemma}
\begin{proof}
The second part of~\cite[Theorem 8.6]{pW82} shows that $h_\mu(f) \geq h(X,\delta)$. Thus, by ~\eqref{eqn:heq2}, if $\delta$ is sufficiently small, $\mu$ is an MME.
\end{proof}

So far we have used $\delta$ to denote the scale at which we work, since the results will be used for various values of $\delta$.  From now on, we fix the scale $\epsilon$ as in Theorem~\ref{thm:main}. We construct $\mu$ as above using $\delta = 6 \eps$ in the selection of the sets $E_n$.  By Proposition \ref{prop:entropy-appeared} and the assumption $\hexp(f,28\epsilon)< \htop(f)$, \eqref{eqn:heq2} is satisfied for this value of $\delta$, so $\mu$ is indeed an MME. 

We now also assume that $\hspec(f,\epsilon) < \htop(f)$, and show that $\mu$ satisfies a certain weak Gibbs property.

\begin{lemma} \label{lem:GibbsG}
For sufficiently large $M$, there exists $K_M>0$ such that for every $(x,n)\in \GGG^M$, we have
\begin{equation}\label{eqn:gibbsG}
\mu(B_n(x, 7\epsilon)) \geq K_M e^{-nh(X)}.
\end{equation}
\end{lemma}
\begin{proof}
Let $M$ be large enough to satisfy the conclusion of Lemma~\ref{lem:lotsinG2} with $\gamma=\frac12$.  Then applying Lemma~\ref{lem:sgeq} with $\delta = \frac {14}6 \epsilon$, we have $\Lamb(\GGG^M,n, 14\epsilon) \geq \frac 12 e^{nh(X)}$ for all $n$, and so there exists $E^G_n \subset \GGG_n^M$ such that $\#E^G_n \geq \frac12 e^{nh(X)}$ and $E^G_n$ is $(n,14\epsilon)$-separated.

To estimate $\mu(B_n(x,7\epsilon))$, we first estimate $\nu_m (f^{-k} B_n(x,7\epsilon))$ for $m \gg n$ and (most values of) $0\leq k\leq m$.  Let $\tau = \tau_M \in \NN$ be provided by the specification property.  Fix $\tau \leq k \leq m-n-\tau$; let $\ell_1 = k-\tau$ and $\ell_2 = m-k-\tau-n$, so that $\ell_1 + \tau + n + \tau + \ell_2 = m$.

Using specification, for each $y\in E^G_{\ell_1}$ and $z\in E^G_{\ell_2}$ there exists
\[
\ph(y,z) \in B_{\ell_1}(y,\epsilon) \cap f^{-\ell_1-\tau} B_n(x,\epsilon) \cap f^{-n-\ell_1-2\tau} B_{\ell_2}(z,\epsilon);
\]
furthermore, because $E^G_n$ is $(n,14\epsilon)$-separated, $\ph$ is injective and $\ph(E^G_{\ell_1} \times E^G_{\ell_2})$ is $(m,12\epsilon)$-separated.

Because the set $E_m$ used to construct $\mu$ is $(m,6\epsilon)$-spanning, we can define $\psi\colon  \ph(E^G_{\ell_1} \times E^G_{\ell_2}) \to E_m$ such that $d_m(w,\psi(w)) < 6\epsilon$, and hence $\psi\circ \ph \colon E^G_{\ell_1}\times E^G_{\ell_2} \to E_m$ is injective.  Furthermore, $\ph(y,z) \in f^{-k}B_n(x,\epsilon)$ for all $y,z$, and hence $\psi\circ\ph(y,z) \in f^{-k}B_n(x,7\epsilon)$.  Therefore
\[
\nu_{m} (f^{-k} B_n(x,7\epsilon)) \geq \frac{(\# E^G_{\ell_1})(\# E^G_{\ell_2})}{\# E_m}
\geq \frac{e^{(\ell_1 + \ell_2)h(X)}}{4C_1 e^{mh(X)}},
\]
where the bound on the denominator comes from Lemma~\ref{lem:sleq2}.  This holds for all $\tau \leq k\leq m-n-\tau$, and averaging over $0\leq k< m$ gives
\[
\mu_m(B_n(x,7\epsilon)) \geq \left(\frac{m-n-2\tau}{4C_1 m}\right) e^{-(2\tau + n)h(X)}.
\]
Sending $m\to\infty$ gives the result with $K_M = (4C_1)^{-1} e^{-2\tau_M h(X)}$.
\end{proof}

Later on, in the proof of ergodicity, we will need a very similar lemma that estimates the measure of a set of points whose trajectories are prescribed not on a single interval, but on two disjoint intervals.

\begin{lemma}\label{lem:Gibbsish}
For sufficiently large $M$, there exists $K_M'>0$ such that for every $(x_1,n_1), (x_2,n_2)\in \GGG^M$ and $q\geq 2\tau_M$, we have
\[
\mu\left(B_{n_1}(x_1,7\epsilon) \cap f^{-n_1-q} B_{n_2}(x_2, 7\epsilon)\right) \geq K_M' e^{-(n_1 + n_2) h(X)}.
\]
\end{lemma}
\begin{proof}
The proof closely resembles that of Lemma \ref{lem:GibbsG}.  Let $M$ be large enough to satisfy the conclusion of Lemma~\ref{lem:lotsinG2} with $\gamma=\frac12$ and let $\tau = \tau_M$.  As before, let $E^G_n\subset \GGG_n^M$ be $(n,14\epsilon)$-separated with $\#E^G_n\geq \frac 12 e^{nh(X)}$.

Given $m \gg n_1 + n_2 + q$ and $\tau \leq k \leq m-n_2-q-n_1-\tau$, let $\ell_1 = k-\tau$ and $\ell_2 = m-n_2-q-n_1-\tau-k$, so that
\[
m = \ell_1 + \tau + n_1 + \tau + \bar q + \tau + n_2 + \tau + \ell_2,
\]
where $\bar q = q-2\tau$.  A similar argument to the one in Lemma \ref{lem:GibbsG} shows that
\begin{align*}
\nu_m (f^{-k} B_{n_1}(x_1,7\epsilon) \cap f^{-k-n_1- q} B_{n_2}(x_2, 7\epsilon)) 
&\geq \frac{(\#E^G_{\ell_1}) (\#E^G_{\bar q}) (\#E^G_{\ell_2})}{\Lamb(X,m,6\epsilon)} \\
&\geq \frac{e^{\ell_1 + \bar q + \ell_2}}{8C_1 e^{mh(X)}}.
\end{align*}
Averaging over $0\leq k < m$ gives
\begin{multline*}
\mu_m(f^{-k} B_{n_1}(x_1,7\epsilon) \cap f^{-k-n_1- q} B_{n_2}(x_2, 7\epsilon)) \\
\geq \left( \frac{ m-n_1-q-n_1-2\tau}{8C_1m} \right) e^{-(4\tau + n_1 + n_2) h(X)},
\end{multline*}
and sending $m\to\infty$ gives the result with $K_M' = (8C_1)^{-1} e^{-4\tau_M h(X)}$.
\end{proof}

\subsection{Adapted partitions and positive measure sets for MMEs} \label{partition}

From now on we need the full strength of the hypotheses in Theorem~\ref{thm:main}.  In particular, we assume that every MME $\nu$ is almost positively expansive at scale $28\epsilon$. We recall that the definition of an adapted partition is given in \S\ref{expansive}.

\begin{lemma}\label{lem:posmeas}
For every $\gamma\in (0,1)$ there exists $C_\gamma>0$ such that if $\nu$ is an MME and $\DDD\subset X\times \NN$ satisfies $\nu(\DDD_n)\geq \gamma$ for every $n$, then $\DDD_n$ has non-empty intersection with at least $C_\gamma e^{nh(X)}$ atoms of $\AAA^n$ for any partition $\AAA^n$ adapted to a maximal $(n,14\epsilon)$-separated set.
\end{lemma}
\begin{proof}
Because $\nu$ is almost positively expansive, we see that for $\nu$-a.e.\ $x\in X$ the following holds:  for all $y\neq x$ there exists $m$ such that $d(f^m y, f^m x) \geq 28\epsilon$.  Writing $m = qn + r$ for $q,r\in \NN$, $r<n$, we see that $d_n((f^n)^q y, (f^n)^q x) \geq 28\epsilon$, and it follows that $\AAA^n$ is generating for $f^n$.  In particular, this implies that
\begin{align*}
h_\nu(f^n) &= h_\nu(f^n,\AAA^n) = \inf_{q\geq 1} \frac 1q H_\nu(\AAA^n\vee (f^n)^{-1} \AAA^n \vee \cdots \vee (f^n)^{-q+1} \AAA^n) \\
& \leq H_\nu(\AAA^n) = \sum_{w\in \AAA^n} -\nu(w) \log \nu(w).
\end{align*}
Without loss of generality, we assume that $\DDD_n$ is a union of atoms of $\AAA^n$, and write this collection of atoms as $\WWW_n$.  Writing $\WWW_n^c = \AAA^n \setminus \WWW_n$ and $\DDD_n^c = X\setminus \DDD_n$, we recall that $h_\nu(f^n) = nh_\nu(f) = nh(X)$ and obtain
\[
nh(X) \leq \sum_{w\in \WWW_n} -\nu(w) \log \nu(w) + \sum_{w\in \WWW_n^c} -\nu(w) \log \nu(w).
\]
Normalising each sum yields
\begin{equation}\label{eqn:nhleq}
\begin{aligned}
nh(X) &\leq \nu(\DDD_n) \left(\sum_{w\in \WWW_n} -\frac{\nu(w)}{\nu(\DDD_n)} \log \left(\frac{\nu(w)}{\nu(\DDD_n)}\right) \right) \\
&\qquad +\nu(\DDD_n^c) \left(\sum_{w\in \WWW_n^c} -\frac{\nu(w)}{\nu(\DDD_n^c)} \log \left(\frac{\nu(w)}{\nu(\DDD_n^c)}\right)\right) \\
&\qquad + (-\nu(\DDD_n)\log \nu(\DDD_n) - \nu(\DDD_n^c) \log \nu(\DDD_n^c)).
\end{aligned}
\end{equation}
Recall that for any non-negative numbers $a_1,\dots a_k$ summing to $1$, we have
\[
\sum_{i=1}^k -a_i \log a_i \leq \log k.
\]
Applying this to the first sum in~\eqref{eqn:nhleq} with the quantities $a_i$ replaced by $\frac{\nu(w_x)}{\nu(\DDD_n)}$, and to the second sum with $a_i$ replaced by $\frac{\nu(w_x)}{\nu(\DDD_n^c)}$, we see that
\[
nh(X) \leq \nu(\DDD_n) \log \#\WWW_n + \nu(\DDD_n^c) \log \#(\WWW_n^c) + H(\nu(\DDD_n)),
\]
where we write $H(t) = -t\log t - (1-t) \log (1-t)$.  Lemma~\ref{lem:sleq2} implies that $\#(\WWW_n^c) \leq \#S_n \leq C_1 e^{nh(X)}$, and so writing $C_2 = \log C_1 + \max_{t\in[0,1]} H(t)$, we have
\begin{align*}
nh(X) &\leq \nu(\DDD_n) \log \#\WWW_n + (1-\nu(\DDD_n)) (\log C_1 + nh(X)) + H(\nu(\DDD_n)) \\
&= \nu(\DDD_n) \log \#\WWW_n + nh(X) - \nu(\DDD_n) (\log C_1 + nh(X))+ C_2,
\end{align*}
which gives
\begin{align*}
\nu(\DDD_n) \log \#\WWW_n &\geq \nu(\DDD_n) (\log C_1 + nh(X)) - C_2, \\
\log \#\WWW_n &\geq \log C_1 + nh(X) - \frac {C_2}{\gamma}.
\end{align*}
Thus it suffices to take $C_\gamma = C_1 e^{-C_2/\gamma}$.
\end{proof}

\subsection{Ergodicity}
Now we consider the MME constructed in \S\ref{sec:Gibbs}, and show that it is ergodic.  This is a consequence of the following proposition.

\begin{proposition}\label{prop:reallyweakmix}
If two measurable sets $P, Q \subset X$ both have positive $\mu$-measure, then $\llim_{n\to\infty} \mu(P \cap \sigma^{-n}(Q)) > 0$.
\end{proposition}
\begin{proof}
In fact, we show that for every $\gamma>0$ there exists $\alpha>0$ such that if $\min \{ \mu(P), \mu(Q) \} \geq 2\gamma$, then
\[
\llim_{n\to\infty} \mu(P \cap \sigma^{-n}(Q)) \geq \alpha.
\]
As in \S \ref{partition}, we take a sequence of partitions $\AAA^n$, each of which is adapted to a maximal $(n, 14\epsilon)$-separated set  $S_n$. Recall that each $x \in S_n$ is identified with a partition element $w_x \in \AAA^n$ such that
\[
B_n(x, 7 \eps) \subset w_x.
\]
Fix $\delta \in (0,\gamma)$.  Because $\Phi_x(28\epsilon)=\{x\}$ for $\mu$-a.e.\ $x\in X$, we can assume WLOG that this holds for every $x\in P,Q$.  Thus $\bigcap_n B_n(x,14\epsilon) = \{x\}$ for all $x\in P,Q$.

In what follows, to simplify notation, we sometimes use the same symbol (such as $U$) to denote both a set of partition elements ($U\subset \AAA^n$) and the union of those partition elements ($U\subset X$).  


\begin{lemma}\label{lem:standard-argument}
Let $\mu$ be a finite Borel measure on a compact metric space $X$ and let $P\subset X$ be measurable.  Suppose $\AAA^n$ is a sequence of partitions such that $\bigcap_n \AAA^n(x) = \{x\}$ for every $x\in P$, and let $\delta>0$.  Then for all sufficiently large $n$ there exists a collection $U \subset \AAA^n$ such that $\mu(U\symdiff P) < \delta$.
\end{lemma}
\begin{proof}
This is a standard argument.  Let $R\subset P$ and $S \subset X\setminus P$ be compact such that $\mu(X \setminus (R \cup S)) < \delta/2$.  Now let $\eta>0$ be such that $d(x,y) > \eta$ for all $x\in R$ and $y\in S$.  For all sufficiently large $n$ there exists $R' \subset R$ such that $\mu(R\setminus R') < \delta/2$ and $\diam \AAA^n(x) < \eta$ for all $x\in R'$.   

Let $U = \{ w \in \AAA^n \mid w\cap R' \neq \emptyset \}$.  Then $R' \subset U \subset X\setminus S$, and the result follows since $U\symdiff P \subset \big(X \setminus (R\cup S)\big) \cup (R \setminus R')$.
\end{proof}

Returning to the proof of Proposition~\ref{prop:reallyweakmix}, we see that Lemma~\ref{lem:standard-argument} allows us to choose, for all sufficiently large $n$, collections $U, V\subset \AAA^n$ such that
\[
\mu (U \symdiff P) < \delta \mbox{ and } \mu (V \symdiff Q) < \delta.
\]
In particular $\mu(U) > \mu (P) - \delta > \gamma$, and $\mu(V) > \mu (Q) - \delta > \gamma$.  Thus, by Lemma \ref{lem:posmeas}, we have
\[
\# U \geq C_\gamma e^{nh(X)} \mbox{ and } \# V \geq C_\gamma e^{nh(X)}.
\]
Let $M$ be such that Lemma~\ref{lem:lotsinG} holds with $\gamma_1 = C_\gamma$ and $\gamma_2 = \frac 12$; write $U' = \{ w_x \in U \mid x\in \GGG^M\}$ and similarly for $V'$.  Then
\begin{equation}\label{eqn:U'V'}
\# U' \geq \frac 12 C_\gamma e^{nh(X)} \mbox{ and } \# V' \geq \frac 12  C_\gamma e^{nh(X)}.
\end{equation}
For every $w_{x_1}\in U'$ and $w_{x_2}\in V'$ and all $m\geq 2\tau + n$, Lemma~\ref{lem:Gibbsish} implies that
\[
\mu(w_{x_1} \cap f^{-m} w_{x_2}) \geq \mu(B_n(x_1,7\epsilon) \cap f^{-m}B_n(x_2,7\epsilon)) 
\geq K_M' e^{-2nh(X)}.
\]
It follows that
\[
\mu(U' \cap f^{-m}V') \geq \left(\frac 14 C_\gamma^2 e^{2nh(X)}\right) (K_M' e^{-2nh(X)}) = \frac{C_\gamma^2 K_M'}4 =:\alpha.
\]
Now since $U\supset U'$ and $V\supset V'$, and since for any $m$,
\[
|\mu(U\cap \sigma^{-m}V) - \mu(P\cap \sigma^{-m}Q)| < \delta,
\]
we have
\[
\llim_{m \rightarrow \infty} \mu(P\cap \sigma^{-m}Q) \geq \llim_{m \rightarrow \infty} \mu(U\cap \sigma^{-m}V) -\delta \geq \alpha - \delta.
\]
Furthermore, since $\delta>0$ was arbitrary, we conclude that
\[
\llim_{m \rightarrow \infty} \mu(P\cap \sigma^{-m}Q) \geq \alpha.\qedhere
\]
\end{proof}

\subsection{Contradiction if there is another mme}

Let $\mu$ be the ergodic MME constructed in the previous sections, and suppose that there exists another ergodic MME $\nu\perp\mu$.  Let $P\subset X$ be such that $\mu(P) = 0$ and $\nu(P)=1$, and let $\AAA^n$ be a sequence of partitions adapted to maximal $(n,14\epsilon)$-separated sets $S_n$.  By Lemma~\ref{lem:standard-argument}, there exists $U_n\subset \AAA^n$ such that $(\mu + \nu)(U_n \symdiff P) \to 0$.  In particular, we have $\nu(U_n) \to 1$ and $\mu(U_n) \to 0$.

Now $\gamma := \inf_n\nu(U_n) > 0$, and so by Lemma~\ref{lem:posmeas}, we have $\#U_n\geq C_\gamma e^{nh(X)}$ for all $n$.  As in the proof of Proposition~\ref{prop:reallyweakmix}, fix $M$ and let $U_n' =\{ w_x \in U_n \mid x\in \GGG^M\}$; by Lemma~\ref{lem:lotsinG}, there exists $M$ such that $\#U_n' \geq \frac 12 C_\gamma e^{nh(X)}$ for all $n$.

Finally, we use the Gibbs property in Lemma~\ref{lem:GibbsG} to observe that
\[
\mu(U_n) \geq \mu(U_n') \geq (\#U_n') K_M e^{-nh(X)} \geq \frac 12 C_\gamma K_M > 0,
\]
which contradicts the fact that $\mu(U_n) \to 0$.  This contradiction implies that any MME $\nu$ is absolutely continuous with respect to $\mu$, and since $\mu$ is ergodic, this in turn implies that $\nu=\mu$, which completes the proof of the theorem.

\section{Application to factors}
We prove that every positively expansive topological factor of any $\beta$-shift has a unique measure of maximal entropy. The key point is that a decomposition $(\PPP, \GGG, \SSS)$  for a system $(X, f)$ induces a decomposition for a factor $(Y, g)$ in a natural way. We demonstrate this phenomenon, which was worked out in the symbolic case in~\cite{beta-factor}, in the special case of factors of the $\beta$-shift.

For a more detailed description of the structure of the $\beta$-shift $(X,\sigma)$, we refer to~\cite{beta-factor} and the references therein; here we only state the properties  we need for the proof.  The key is that there exists a distinguished sequence $w\in \Sigma^+ := \{0,1,\dots,b\}^\NN$ (the $\beta$-expansion of $1$) such that $z\in \Sigma^+$ is in $X$ if and only if $\sigma^n z\preceq w$ for all $n\geq 0$, where $\preceq$ denotes the lexicographic order. 

The language $\LLL$ of the $\beta$-shift is the collection of all finite words that appear in sequences $z\in X$.  We recall the decomposition of $\LLL$ given in~\cite{beta-factor}:  $\tilde\CCC^s = \{w_1 \cdots w_n \mid n\geq 1\}$ is the collection of prefixes of $w$, and $\tilde\GGG$ is the collection of words in $\LLL$ that do not end with any prefix of $w$.  That is,
\[
\tilde\GGG = \{ v_1 \dots v_m \in \LLL \mid v_{m-n+1} \cdots v_{m-1} v_m \neq w_1 \cdots w_n \text{ for all } 1\leq n\leq m \}.
\]
Every word $v\in \LLL$ can be decomposed as $v=v_1 \cdots v_m w_1 \cdots w_n$, where $v_1 \cdots v_m \in \tilde\GGG$.  We will use the following properties of words in $\tilde\GGG$, which are easily seen from the order-theoretic characterisation above or from the graph presentation of $X$ described in \cite{beta-factor}.
\begin{enumerate}
\item \label{append} Given any $v\in \LLL$, the word $v0^k$ is in $\tilde\GGG$ for all sufficiently large $k$.
\item \label{exact} For any $u \in \tilde\GGG$ and any $v \in \LLL$, the 
word $uv \in\LLL$. 
\end{enumerate}


Let $(Y, g)$ be a topological factor of $(X, \sigma)$, and let $\pi\colon X \mapsto Y$ be the factor map.

\begin{lemma}
$(Y,g)$ has positive topological entropy or $Y$ is a single point.
\end{lemma} 
\begin{proof}
Assume that $Y$ is not a single point. Then there exists $x,y, \delta_0$ so that $d(x,y) = 3 \delta_0$ (where $d$ is the metric on $Y$). By uniform continuity of $\pi$, there exists $\delta$ so that $\pi(B(z, \delta)) \subset B(\pi z, \delta_0)$ for all $z \in X$.  Fix $\tilde x \in \pi^{-1}(x)$ and $\tilde y\in \pi^{-1}(y)$.  

For sufficiently large $n$, there are words $v_1,v_2\in \tilde\GGG$ of length $n$ such that $[v_1] \subset B(\tilde x, \delta)$ and $[v_2] \subset B(\tilde y, \delta)$.  To see this, take prefixes of $\tilde x$ and $\tilde y$ with sufficient length that the first condition holds, then append the symbol $0$ to these prefixes enough times that the resulting words have the same length and are both in $\tilde\GGG$, using property \eqref{append}.


Now, for any $\underline i = i_1 \ldots i_m \in \{1, 2\}^m$, we can choose $y_{\underline i} \in [v_{i_1} v_{i_2} \ldots v_{i_m}]$ using property \eqref{exact}, and define $x_{\underline i} = \pi y_{\underline i}$. The points $\{x_{\underline i}\}$ form a $(nm, \delta_0)$-separated set in $Y$. Thus $\Lamb(Y, nm, \delta_0) \geq 2^m$ for every $m$. This shows that the topological entropy of $(Y, g)$ is positive.
\end{proof}

\begin{proposition}\label{prop:factor-spec}
$\hspec(f,\epsilon) = 0$ for every $\epsilon>0$.
\end{proposition}
\begin{proof}
We obtain a decomposition $(\PPP, \GGG, \SSS)$ for $(Y,g)$ from the decomposition $\LLL = \tilde\GGG \tilde\CCC^s$, then show that $\GGG^M$ has specification for every $\epsilon>0$.  (Of course the transition time $\tau$ depends on $\epsilon$ and $M$.)  let
\begin{align*}
\GGG &= \{ (x, n) \mid x = \pi \tilde x \mbox{ for some } \tilde x \in X \mbox{ such that } \tilde x_1 \ldots \tilde x_n \in \tilde\GGG \}, \\
\SSS &= \{ (x, n) \mid x = \pi \tilde x \mbox{ for some } \tilde x \in X \mbox{ such that } \tilde x_1 \ldots \tilde x_n \in \tilde\CCC^s \}.
\end{align*}
Setting $\PPP = \emptyset$, this defines a decomposition for $Y$. 

\begin{lemma}\label{lem:factor-spec}
$\GGG^M$ has the specification property for all $\epsilon>0$ and $M\in \NN$.
\end{lemma}
\begin{proof}
Fix $\epsilon$ and take $\delta$ so that $\pi(B(x, \delta)) \subset B(\pi x, \epsilon)$ for all $x \in X$.  There exists $m= m(\delta)$ so that for all $\tilde y \in X$, we have $[\tilde y_1 \cdots \tilde y_m] \in B(y, \delta)$.  Then for every $\tilde z\in [\tilde y_1 \cdots \tilde y_{m+n}]$, we have $d_n(\tilde y,\tilde z) < \delta$, and so for $y=\pi\tilde y$ and $z=\pi\tilde z$ we have $d_n(y,z) < \epsilon$, where $d_n(y,z) := \max_{0\leq k\leq n} d(g^n(y), g^n(z))$ is the Bowen metric.


Given any collection $\{(x^1, n_1), \ldots, (x^k, n_k) \} \subset \GGG^M$, 
let $\tilde x^i\in \pi^{-1} x_i$ be such that $\tilde x^i_1 \cdots \tilde x^i_{n_i+m} \in \tilde \GGG^M$.  Since $\tilde\GGG^M$ has specification with  gap length $\tau_M$, we can find connecting words $v^i$ of length $\tau_M$ such that the cylinder
\[
[\tilde x^1_1 \ldots \tilde x^1_{n_1+m} v^1\tilde x^2_1 \ldots \tilde x^2_{n_2+m}v^2 \ldots \tilde x^k_1 \ldots \tilde x^1_{n_k+m}]
\]
is non-empty.  In particular, any point $z$ in the image of this cylinder under $\pi$ verifies the specification property for the collection $\{(x^1, n_1), \ldots, (x^k, n_k) \}$, where the gap length is $m + \tau_M$. 
\end{proof}

\begin{lemma} \label{lem:factor-ent}
$h(\SSS) = 0$.
\end{lemma}
\begin{proof}
Let $w= w_1 w_2 \ldots$ be the $\beta$-expansion of $1$, and let $z = \pi (w)$.  Then, by definition,
\[
\SSS \subset \{ (x, n) \mid x = \pi \tilde x, ~\tilde x_1 \ldots \tilde x_n = w_1 \ldots w_n \}.
\]
Take $\delta$ so that $\pi(B(x, \delta)) \subset B(\pi x, \epsilon)$ for all $x \in X$. There exists $k$ so that if $y \in \SSS_{n+k}$, then there exists $\tilde y \in \pi^{-1}(y) \cap B_n (w, \delta)$, and thus $y \in B_n(z, \epsilon)$. Thus, for all $n$, $\SSS_{n+k} \subset B_n(z, \epsilon)$. In particular, $\Lamb(\SSS_{n+k}, n+k, \epsilon) \leq \Lamb(X, k, \epsilon)$, and $h(\SSS, \epsilon) \leq \lim_{n \rightarrow \infty}\frac{1}{n} \log \Lamb(X, k, \epsilon) = 0$. Since $\epsilon$ was arbitrary, we are done.
\end{proof}
Proposition~\ref{prop:factor-spec} follows from Lemmas~\ref{lem:factor-spec} and \ref{lem:factor-ent}.
\end{proof}

Using Proposition~\ref{prop:factor-spec}, we can apply Theorem~\ref{thm:main} to $(Y, g)$ as long as $\hexp(g,\epsilon) < h(Y)$ for some $\epsilon$. In particular, we can apply Theorem~\ref{thm:main} when $(Y, g)$ is positively expansive.

\bibliographystyle{amsalpha}
\bibliography{h-expansive-mme}

\providecommand{\bysame}{\leavevmode\hbox to3em{\hrulefill}\thinspace}
\providecommand{\MR}{\relax\ifhmode\unskip\space\fi MR }
\providecommand{\MRhref}[2]{%
  \href{http://www.ams.org/mathscinet-getitem?mr=#1}{#2}
}
\providecommand{\href}[2]{#2}
\begin{thebibliography}{Bow75}

\bibitem[BF11]{BF11}
Jerome Buzzi and Todd Fisher, \emph{Entropic stability beyond partial
  hyperbolicity}, Preprint, 27 pages, \texttt{arXiv:1103.2707}, 2011.

\bibitem[BK83]{BK83}
M.~Brin and A.~Katok, \emph{On local entropy}, Geometric dynamics ({R}io de
  {J}aneiro, 1981), Lecture Notes in Math., vol. 1007, Springer, Berlin, 1983,
  pp.~30--38. \MR{730261 (85c:58063)}

\bibitem[Bow75]{rB74}
Rufus Bowen, \emph{Some systems with unique equilibrium states}, Math. Systems
  Theory \textbf{8} (1974/75), no.~3, 193--202. \MR{0399413 (53 \#3257)}

\bibitem[Buz97]{jB97}
J.~Buzzi, \emph{Intrinsic ergodicity of smooth interval maps}, Israel J. Math.
  \textbf{100} (1997), no.~1, 125--161.

\bibitem[Cli13]{pressure-spectrum}
Vaughn Climenhaga, \emph{Topological pressure of simultaneous level sets},
  Nonlinearity \textbf{26} (2013), no.~1, 241--268.

\bibitem[CT]{beta-factor}
Vaughn Climenhaga and Daniel~J. Thompson, \emph{Intrinsic ergodicity beyond
  specification: $\beta$-shifts, {$S$}-gap shifts, and their factors}, Israel
  J. Math., 26 pages, to appear, \texttt{arXiv:1011.2780}.

\bibitem[Hof79]{fH79}
F.~Hofbauer, \emph{On intrinsic ergodicity of piecewise monotonic
  transformations with positive entropy}, Israel J. Math. \textbf{34} (1979),
  no.~3, 213--237.

\bibitem[Hof81]{fH81}
\bysame, \emph{{On intrinsic ergodicity of piecewise monotonic transformations
  with positive entropy II}}, Israel J. Math. \textbf{38} (1981), no.~1-2,
  107--115.

\bibitem[LS04]{LS}
E.~Lindenstrauss and K.~Schmidt, \emph{Invariant sets and measures of
  nonexpansive group automorphisms}, Israel J. Math. \textbf{144} (2004),
  29--60.

\bibitem[Pes97]{yP97}
Yakov~B. Pesin, \emph{Dimension theory in dynamical systems}, Chicago Lectures
  in Mathematics, University of Chicago Press, Chicago, IL, 1997, Contemporary
  views and applications. \MR{1489237 (99b:58003)}

\bibitem[Sch00]{kS00}
Klaus Schmidt, \emph{Algebraic coding of expansive group automorphisms and
  two-sided beta-shifts}, Monatsh. Math. \textbf{129} (2000), no.~1, 37--61.
  \MR{1741033 (2001f:54043)}

\bibitem[Sid03]{nS03}
N.~Sidorov, \emph{Arithmetic dynamics}, {Topics in Dynamics and Ergodic
  Theory}, LMS Lecture Notes Ser., vol. 310, 2003, pp.~145--189.

\bibitem[Wal82]{pW82}
Peter Walters, \emph{An introduction to ergodic theory}, Graduate Texts in
  Mathematics, vol.~79, Springer-Verlag, New York, 1982. \MR{648108
  (84e:28017)}

\bibitem[Wei70]{bW70}
B.~Weiss, \emph{Intrinsically ergodic systems}, Bull. Amer. Math. Soc.
  \textbf{76} (1970), no.~6, 1266--1269.

\end{thebibliography}

\end{document}